\documentclass[12pt]{amsart}

\usepackage[margin=1.1in]{geometry}
\usepackage{amsmath,amssymb,amsthm}
\usepackage{booktabs}
\usepackage{microtype}
\usepackage[colorlinks=true,linkcolor=blue,citecolor=blue,urlcolor=blue]{hyperref}

\theoremstyle{plain}
\newtheorem{theorem}{Theorem}[section]
\newtheorem{lemma}[theorem]{Lemma}

\theoremstyle{definition}
\newtheorem{definition}{Definition}[section]
\newtheorem{example}[theorem]{Example}
\theoremstyle{remark}
\newtheorem{remark}[theorem]{Remark}

\newcommand{\alg}[1]{\mathfrak{#1}}
\newcommand{\RR}{\mathbb{R}}
\newcommand{\CC}{\mathbb{C}}
\newcommand{\HH}{\mathbb{H}}
\newcommand{\OO}{\mathbb{O}}
\newcommand{\Idm}{\mathrm{Id}}
\newcommand{\Ker}{\operatorname{Ker}}
\newcommand{\diag}{\operatorname{diag}}
\newcommand{\imd}{\operatorname{Im}}

\begin{document}

\title{On the Structure of Sub-Riemannian Geodesic Orbit Manifolds}

\author{Huihui An}
\address{School of Mathematics, Liaoning Normal University, Dalian, Liaoning Province, 116029,
People's Republic of China}
\email{anhh@lnnu.edu.cn}

\author{Zaili Yan$^{*}$}
\address{School of Mathematics and Statistics, Ningbo University, Ningbo, Zhejiang Province,
315211, People's Republic of China}
\email{yanzaili@nbu.edu.cn}
\thanks{Z. Yan is the corresponding author.}

\author{Hui Zhang}
\address{School of Mathematics, Southeast University, Nanjing 211189, People's Republic of China}
\email{huizhang@mail.nankai.edu.cn}

\author{Shaoxiang Zhang}
\address{College of Mathematics and Systems Science, Shandong University of Science and
Technology, Qingdao 266590, People's Republic of China}
\email{zhangshaoxiang93@163.com}

\date{}

\begin{abstract}
We study homogeneous sub-Riemannian manifolds whose normal extremals are orbits of one-parameter
subgroups of the acting group. Our first main result shows that every sub-Riemannian geodesic orbit
manifold satisfies the Goh condition and consequently admits no strictly abnormal length minimizers.
Our second main result classifies all compact, connected, simply connected sub-Riemannian geodesic
orbit manifolds with nonabelian simple isotropy group, thereby extending the Riemannian
classification of Chen, Nikolayevsky and Nikonorov.
\end{abstract}

\maketitle

\tableofcontents

\section{Introduction and main results}\label{sec:intro}

A sub-Riemannian structure on a connected smooth manifold $M$ consists of a smooth
bracket-generating subbundle $\mathcal{D}\subset TM$, together with a smooth positive definite
quadratic form $Q$ (a sub-Riemannian metric) defined on the fibers of $\mathcal{D}$. Bracket
generation means that, for every $p\in M$, the local smooth sections of $\mathcal{D}$ near $p$,
together with all their iterated Lie brackets, span the tangent space $T_pM$. By the
Chow--Rashevskii theorem \cite{Chow,Rashevskii}, any two points of $M$ can be joined by an
admissible curve, that is, an absolutely continuous curve $\gamma$ with
$\dot{\gamma}(t)\in\mathcal{D}_{\gamma(t)}$ for almost every $t$; the sub-Riemannian distance is
the infimum of the $Q$-lengths of admissible curves.

A geodesic is a curve that is locally a constant-speed length minimizer. By the Pontryagin Maximum
Principle, length minimizers lift to extremals in the cotangent bundle $T^*M$. The normal
extremals are precisely the integral curves of the Hamiltonian vector field of
\begin{equation}\label{eq:ham}
H(\lambda)=\tfrac12\,Q^{*}(\lambda|_{\mathcal{D}},\lambda|_{\mathcal{D}}),\qquad \lambda\in T^*M,
\end{equation}
where $Q^{*}$ is the dual metric on the dual bundle $\mathcal{D}^{*}$. The projections of normal
extremals are the normal geodesics. A geodesic that admits no normal lift is called
strictly abnormal. It is a classical fact that the Goh condition
\begin{equation}\label{eq:goh}
\mathcal{D}_p+[\mathcal{D},\mathcal{D}]_p=T_pM\qquad(p\in M)
\end{equation}
forces every geodesic to be normal; see \cite[Corollary 12.14]{Agrachev}. Here
$[\mathcal{D},\mathcal{D}]_p$ denotes the span of $[Y_1,Y_2](p)$ for local smooth sections
$Y_1,Y_2$ of $\mathcal{D}$.

Let $(M,\mathcal{D},Q)$ be a homogeneous sub-Riemannian manifold, and let $G$ be a connected Lie
group acting effectively and transitively on $M$ by smooth isometries. Fix a point $p\in M$ and
let $H$ be the isotropy subgroup of $G$ at $p$, so that $M=G/H$ and $H$ is compact; denote by
$\mathfrak{g}$ and $\mathfrak{h}$ the Lie algebras of $G$ and $H$, respectively. Following
Berestovskii \cite{Berestovskii}, the structure is encoded by the following data (see
\cite{Berestovskii,YAD}): a reductive decomposition $\mathfrak{g}=\mathfrak{h}\oplus\mathfrak{m}$
with $\mathfrak{m}$ an $\mathrm{Ad}(H)$-invariant complement, an $\mathrm{Ad}(H)$-invariant
bracket-generating subspace $\Delta\subset\mathfrak{m}$ (meaning that $\mathfrak{h}\oplus\Delta$
generates $\mathfrak{g}$ as a Lie algebra), and an $\mathrm{Ad}(H)$-invariant inner product
$\langle\cdot,\cdot\rangle_\Delta$ on $\Delta$. The distribution is
$\mathcal{D}=\bigcup_{g\in G}g_*(\Delta)$ and $Q$ is determined by $\langle\cdot,\cdot\rangle_\Delta$.
We denote this homogeneous sub-Riemannian manifold by
$(G/H,\Delta,\langle\cdot,\cdot\rangle_\Delta)$.

\begin{definition}[Sub-Riemannian geodesic orbit manifolds]\label{def:GO}
A connected homogeneous sub-Riemannian manifold $(M,\mathcal{D},Q)$ is geodesic orbit
(abbreviated GO) with respect to $G$ if every normal extremal in $T^*M$ is an orbit of a
one-parameter subgroup of $G$ under the cotangent lift of the $G$-action. In this case every
normal geodesic is an orbit $\exp(tX)\cdot p$ of a one-parameter subgroup, and a nonzero vector
$X\in\mathfrak{g}$ with this property is called a geodesic vector.
\end{definition}

This notion generalizes the classical Riemannian geodesic orbit property
\cite{Gordon96,KowalskiVanhecke}; we refer the reader to
\cite{Arvanitoyeorgos,BerestovskiiNikonorov} for further background and recent results.

Our first main result establishes a two-step generation property; it answers in the affirmative a
natural question concerning the relationship between the GO property and the Goh condition posed
in \cite[Remark 1.3]{YAD}.

\begin{theorem}[Two-step generation]\label{thm:twostep}
If $(M,\mathcal{D},Q)$ is geodesic orbit with respect to $G$, then
\[ \mathcal{D}_p+[\mathcal{D},\mathcal{D}]_p=T_pM\qquad\text{for every }p\in M. \]
Equivalently, for an $\mathrm{Ad}(H)$-invariant reductive decomposition
$\mathfrak{g}=\mathfrak{h}\oplus\mathfrak{m}$ and the horizontal subspace
$\Delta=\mathcal{D}_{eH}\subset\mathfrak{m}$,
\begin{equation}\label{eq:twostep}
\mathfrak{g}=\mathfrak{h}+\Delta+[\Delta,\Delta],
\end{equation}
where $[\Delta,\Delta]$ denotes the linear span of the indicated Lie brackets. Consequently, every
constant-speed locally length-minimizing geodesic on a sub-Riemannian geodesic orbit manifold has
a normal lift; in other words, there are no strictly abnormal length minimizers.
\end{theorem}

In \cite{ChenNN}, Chen, Nikolayevsky and Nikonorov classified the compact, simply connected
Riemannian geodesic orbit manifolds with simple isotropy group (see \cite[Thms.~1--2]{ChenNN}).
Our second main result extends their classification to the sub-Riemannian setting.

\begin{theorem}\label{thm:main}
Let $(M,\mathcal{D},Q)=(G/H,\Delta,\langle\cdot,\cdot\rangle_\Delta)$ be a compact, connected,
simply connected homogeneous sub-Riemannian manifold, where the action of $G$ is almost effective
and $H$ is a connected simple Lie group. Then every proper sub-Riemannian $G$-GO structure
belongs to one of the following two families, and every structure listed is GO and bracket
generating.
\begin{enumerate}\renewcommand{\labelenumi}{\upshape(\Alph{enumi})}
\item Take one of the simple-group structures in Tables~\ref{tab:two} and~\ref{tab:rem},
optionally in product with a compact simply connected semisimple group carrying any positive
bi-invariant Riemannian metric.
\item Let $K$ be a compact simply connected semisimple Lie group, $H\subset K$ a connected simple
subgroup that is not normal, and $b$ any positive bi-invariant metric on $K$. On $K$ take the
left-invariant structure
\[ \Delta=\mathfrak{h}^{\perp_b}\subset\mathfrak{k},\qquad Q=b|_{\Delta}, \]
where $\mathfrak{h}^{\perp_b}$ is the $b$-orthogonal complement of the Lie algebra $\mathfrak{h}$
of $H$ in the Lie algebra $\mathfrak{k}$ of $K$. Its acting group is $K\times H$ by left and right
translations, with isotropy group $\diag H$. Its corank is $\dim H$.
\end{enumerate}
\end{theorem}

In the last case (10) of Table~\ref{tab:rem}, $L_x$ and $R_x$ denote left and right multiplication
on the real division algebra  of octonions $\OO$.

\begin{table}[htbp]
\centering
\caption{Two-summand cases. In every row $\Delta=\alg{m}_1$ and $A=\beta\Idm|_{\alg{m}_1}$,
$\beta>0$. The other summand is $\alg{m}_2=\alg{m}\ominus\alg{m}_1$. Row numbers agree with
\cite[Table~1]{ChenNN}.}
\label{tab:two}
\begin{tabular}{@{}l l p{2.7in}@{}}
\toprule
Row & $G/H$ & Horizontal subspace $\alg{m}_1$\\
\midrule
(1) & $\mathrm{Spin}(9)/\mathrm{Spin}(7)$ & $\alg{so}(9)\ominus\alg{so}(8)$ (dimension 8)\\
(2) & $\mathrm{Spin}(10)/\mathrm{Spin}(7)$ & $\alg{so}(10)\ominus\alg{so}(8)$ (dimension 17)\\
(3) & $\mathrm{Spin}(11)/\mathrm{Spin}(7)$ & $\alg{so}(11)\ominus\alg{so}(8)$ (dimension 27)\\
(4) & $E_6/\mathrm{Spin}(10)$ & $\alg{e}_6\ominus(\alg{so}(10)\oplus\RR)$ (dimension 32)\\
(5) & $\mathrm{SU}(n{+}p)/\mathrm{SU}(n)$, $n\ge2$, $1\le p\le n-1$ &
$\alg{su}(p)\oplus\CC^{p\times n}$, where $\CC^{p\times n}$ denotes the off-diagonal block (real
dimension $p^2-1+2np$)\\
(6$_2$) & $\mathrm{Spin}(2n{+}1)/\mathrm{SU}(n)$, $n$ even, $n\ge4$ &
$\alg{so}(2n+1)\ominus\alg{so}(2n)\simeq\CC^n$ (real dimension $2n$)\\
(7) & $\mathrm{Spin}(4n{+}2)/\mathrm{SU}(2n{+}1)$, $n\ge2$ &
$\alg{so}(4n+2)\ominus\alg{u}(2n+1)$ (dimension $2n(2n+1)$)\\
(8) & $\mathrm{Sp}(n{+}1)/\mathrm{Sp}(n)$, $n\ge1$ &
$\alg{sp}(n+1)\ominus(\alg{sp}(n)\oplus\alg{sp}(1))\simeq\HH^n$\\
(11) & $\mathrm{Spin}(9)/G_2$ & $\alg{so}(9)\ominus\alg{so}(7)$ (dimension 15)\\
\bottomrule
\end{tabular}
\end{table}

\begin{table}[htbp]
\centering
\caption{The remaining simple-group cases. All parameter inequalities are part of the
classification.}
\label{tab:rem}
\begin{tabular}{@{}l l p{2.7in}@{}}
\toprule
Row & $G/H$ & Summands and cometric\\
\midrule
(6$_1$) & $\mathrm{Spin}(2n{+}1)/\mathrm{SU}(n)$, $n$ odd, $n\ge3$ &
$\alg{m}_1=\alg{so}(2n{+}1)\ominus\alg{so}(2n)$, $\alg{m}_2=\alg{so}(2n)\ominus\alg{u}(n)$,
$\alg{m}_3=\alg{u}(n)\ominus\alg{su}(n)$. The coefficients are $(\beta,0,\beta/n)$, $\beta>0$;
thus $\Delta=\alg{m}_1\oplus\alg{m}_3$.\\[4pt]
(9) & $\mathrm{SU}(2n{+}1)/\mathrm{Sp}(n)$, $n\ge2$ &
$\alg{m}_1=\alg{su}(2n{+}1)\ominus\alg{u}(2n)$, $\alg{m}_2=\alg{su}(2n)\ominus\alg{sp}(n)$,
$\alg{m}_3=\alg{u}(2n)\ominus\alg{su}(2n)$. The coefficients satisfy $\beta_1>0$,
$\beta_2,\beta_3\ge0$, $\beta_2\beta_3=0$.\\[4pt]
(10) & $\mathrm{Spin}(8)/G_2$ &
Identify $\alg{so}(8)=\mathrm{Der}(\OO)\oplus L_{\mathrm{Im}\,\OO}\oplus R_{\mathrm{Im}\,\OO}$ and
put $V_{a,b}=\{aL_x+bR_x:x\in\mathrm{Im}\,\OO\}$. Then $A=\beta\Idm|_{V_{a,b}}$, $\beta>0$, where
$[a:b]\in\RR P^1$ and $(a+b)(2a-b)(a-2b)\neq0$.\\
\bottomrule
\end{tabular}
\end{table}

\section{The geodesic lemma and the proof of Theorem~\ref{thm:twostep}}\label{sec:geod}

Throughout this section, $(G/H,\Delta,\langle\cdot,\cdot\rangle_\Delta)$ is a connected
homogeneous sub-Riemannian manifold. Since $H$ is compact, $\langle\cdot,\cdot\rangle_\Delta$
extends to an $\mathrm{Ad}(H)$-invariant inner product $\langle\cdot,\cdot\rangle$ on
$\alg{m}$; such an extension is called a Riemannian extension of
$\langle\cdot,\cdot\rangle_\Delta$, and $\Delta^{\perp}$ denotes the
$\langle\cdot,\cdot\rangle$-orthogonal complement of $\Delta$ in $\alg{m}$. For $X,Y\in\alg{g}$
we write $[X,Y]_{\alg{m}}$ for the projection of the Lie bracket $[X,Y]$ to $\alg{m}$ along
$\alg{h}$.
We now recall the  geodesic lemma and its covector form.
The cotangent-lift formulation below follows T\'oth \cite{Toth} and Podobryaev \cite{Podobryaev},
and was cast in the following algebraic form by Yan--An--Deng \cite{YAD}; it is the
sub-Riemannian analogue of the Kowalski--Vanhecke geodesic lemma \cite{KowalskiVanhecke}.

\begin{lemma}[Geodesic lemma]\label{lem:geod}
Let $(G/H,\Delta,\langle\cdot,\cdot\rangle_\Delta)$ be a connected homogeneous sub-Riemannian
manifold, let $\langle\cdot,\cdot\rangle$ be an $\mathrm{Ad}(H)$-invariant Riemannian extension of
$\langle\cdot,\cdot\rangle_\Delta$ on $\alg{m}$, and let $\Delta^{\perp}$ be the
$\langle\cdot,\cdot\rangle$-orthogonal complement of $\Delta$ in $\alg{m}$. Then
$(G/H,\Delta,\langle\cdot,\cdot\rangle_\Delta)$ is geodesic orbit with respect to $G$ if and only
if, for every $X\in\Delta$ and every $Y\in\Delta^{\perp}$, there exists $Z\in\alg{h}$ such that
\begin{equation}\label{eq:geodlemma}
\langle X+Y,[X+Z,W]_{\alg{m}}\rangle=0\qquad\text{for all }W\in\alg{g}.
\end{equation}
In this case the normal geodesic through the base point $eH$ is of the form
$\exp t(X+Z)\cdot eH$, $t\in\RR$.
\end{lemma}

We next present a covector reformulation of \eqref{eq:geodlemma}, which is better suited to the
proofs in Section~\ref{sec:compact}. Let
\[ \alg{h}^0:=\{\lambda\in\alg{g}^*:\lambda|_{\alg{h}}=0\} \]
be the annihilator of $\alg{h}$ in $\alg{g}^*$, identified with the cotangent space
$T^*_{eH}(G/H)$ at the base point. Define the linear cometric map $\mathcal{A}:\alg{h}^0\to\Delta$ by
\begin{equation}\label{eq:cometric}
Q(\mathcal{A}\lambda,X)=\lambda(X),\qquad X\in\Delta,
\end{equation}
so that $A\lambda$ is the horizontal velocity of the normal extremal with initial covector
$\lambda$. The Hamiltonian \eqref{eq:ham} restricts to the quadratic form
\[ H(\lambda)=\tfrac12\,\lambda(\mathcal{A}\lambda)=\tfrac12\,Q(\mathcal{A}\lambda,\mathcal{A}\lambda),\qquad
\lambda\in\alg{h}^0. \]

\begin{lemma}\label{lem:covector}
$(G/H,\Delta,\langle\cdot,\cdot\rangle_\Delta)$ is geodesic orbit with respect to $G$ if and only
if, for every $\lambda\in\alg{h}^0$, there exists $Z\in\alg{h}$ such that
\begin{equation}\label{eq:covector}
\lambda\big([\mathcal{A}\lambda+Z,W]\big)=0\qquad\text{for all }W\in\alg{g}.
\end{equation}
\end{lemma}

\begin{proof}
Writing $\lambda\in\alg{h}^0$ as $\lambda=\langle X+Y,\cdot\rangle$ with $X=\mathcal{A}\lambda\in\Delta$
its horizontal part and $Y\in\Delta^{\perp}$ its vertical part, equation \eqref{eq:geodlemma} is
equivalent to \eqref{eq:covector}.
\end{proof}

\begin{proof}[Proof of Theorem~\ref{thm:twostep}]
Since the sub-Riemannian structure is homogeneous, it suffices to verify \eqref{eq:twostep} at the
base point $p=eH$; the passage from $\alg{g}=\alg{h}+\Delta+[\Delta,\Delta]$ to
$\mathcal{D}_p+[\mathcal{D},\mathcal{D}]_p=T_pM$ for arbitrary $p$ is the
homogeneous-distribution correspondence of Berestovskii \cite{Berestovskii}, followed by
translation by $G$.

Let $\pi:\alg{g}\to\Delta^{\perp}$ be the natural projection along $\alg{h}\oplus\Delta$. Define
\[ V:=\operatorname{span}\{\pi([X_1,X_2]):X_1,X_2\in\Delta\},\qquad
U:=V^{\perp}\cap\Delta^{\perp}. \]
Both $V$ and $U$ are $\mathrm{Ad}(H)$-invariant. We claim that
\begin{equation}\label{eq:claim}
[\Delta,V]\subset\alg{h}\oplus\Delta\oplus V.
\end{equation}
Fix $X\in\Delta$, $v\in V$ and $u\in U$. By Lemma~\ref{lem:geod}, there exists $Z\in\alg{h}$ such
that
\[ \langle X+u,[X+Z,W]_{\alg{m}}\rangle=0\qquad\forall\,W\in\alg{g}. \]
Taking $W=v$ gives
\[ 0=\langle X+u,[X+Z,v]_{\alg{m}}\rangle=\langle X+u,[X,v]_{\alg{m}}\rangle
+\langle X+u,[Z,v]_{\alg{m}}\rangle. \]
Since $[\alg{h},V]\subset V\subset\alg{m}$, we have $[Z,v]=[Z,v]_{\alg{m}}\in V$. Hence
$\langle X+u,[Z,v]_{\alg{m}}\rangle=0$, and we obtain
\begin{equation}\label{eq:step8}
\langle X+u,[X,v]_{\alg{m}}\rangle=0.
\end{equation}
This identity holds for every $u\in U$, including $u=0$. Subtracting the case $u=0$ from
\eqref{eq:step8} yields
\[ \langle u,[X,v]_{\alg{m}}\rangle=0,\qquad\forall\,u\in U. \]
Thus $[X,v]_{\alg{m}}$ has no $U$-component. Since
$\alg{m}=\Delta\oplus\Delta^{\perp}=\Delta\oplus V\oplus U$, this means
$[X,v]_{\alg{m}}\in\Delta\oplus V$, and therefore $[X,v]\in\alg{h}\oplus\Delta\oplus V$, proving
\eqref{eq:claim}. Set
\[ F:=\alg{h}\oplus\Delta\oplus V. \]
By construction,
\begin{equation}\label{eq:Feq}
F=\alg{h}+\Delta+[\Delta,\Delta].
\end{equation}
Indeed, $[\Delta,\Delta]\subset\alg{h}\oplus\Delta\oplus V=F$ and
$\alg{h}+\Delta+[\Delta,\Delta]\subset F$, while conversely $V\subset[\Delta,\Delta]+\alg{h}+\Delta$,
so that $F\subset\alg{h}+\Delta+[\Delta,\Delta]$. Note also that
\begin{equation}\label{eq:Finv}
[\alg{h}\oplus\Delta,F]\subset F.
\end{equation}
Now $F$ contains $\alg{h}\oplus\Delta$, and we prove that $F$ contains every iterated Lie bracket
of elements of $\alg{h}\oplus\Delta$. By the Jacobi identity and antisymmetry, every such iterated
Lie bracket is a linear combination of left-nested brackets
$[y_1,[y_2,[\cdots,[y_{k-1},y_k]\cdots]]]$ with $y_i\in\alg{h}\oplus\Delta$. We argue by induction
on $k$. For $k=1$ we have $y_1\in\alg{h}\oplus\Delta\subset F$. For $k\ge2$, the induction
hypothesis gives $[y_2,[\cdots,[y_{k-1},y_k]\cdots]]\in F$, and since $y_1\in\alg{h}\oplus\Delta$,
equation \eqref{eq:Finv} gives $[y_1,[y_2,[\cdots]]]\in F$. Hence all iterated Lie brackets lie in
$F$, and so $F=\alg{g}$. By \eqref{eq:Feq} we obtain
$\alg{g}=\alg{h}+\Delta+[\Delta,\Delta]$, which is \eqref{eq:twostep}. The final assertion follows
from the Goh condition \eqref{eq:goh}.
\end{proof}

We conclude this section by describing the nilradical $\alg{n}$ of $\alg{g}$.

\begin{lemma}\label{lem:abelian}
Suppose that $(G/H,\Delta,\langle\cdot,\cdot\rangle_\Delta)$ is geodesic orbit with respect to
$G$. Every abelian ideal $\alg{a}$ of $\alg{g}$ has zero intersection with $\alg{h}$. Every
$\mathrm{Ad}(H)$-invariant positive definite inner product on $\alg{a}$ is
$\mathrm{ad}(\alg{g})$-invariant. In particular, $[\alg{n},\alg{a}]=0$.
\end{lemma}

\begin{proof}
For $X\in\alg{a}$, one has $(\mathrm{ad}\,X)^2=0$ because $\alg{a}$ is an abelian ideal. If also
$X\in\alg{h}$, then the compactness of $H$ provides an $\mathrm{Ad}(H)$-invariant positive
definite inner product on $\alg{g}$ for which $\mathrm{ad}\,X$ is skew-adjoint. Hence
$\mathrm{ad}\,X$ is both nilpotent and skew-adjoint, so $\mathrm{ad}\,X=0$. Effectiveness of the
$G$-action implies $\alg{h}\cap\alg{z}(\alg{g})=0$, where $\alg{z}(\alg{g})$ is the center, so
$X=0$. This proves $\alg{a}\cap\alg{h}=0$.

Fix an $\mathrm{Ad}(H)$-invariant inner product $b$ on $\alg{a}$. Take $v\in\alg{a}$ and let
$\mu=b(v,\cdot)\in\alg{a}^*$. Since $\alg{a}\cap\alg{h}=0$, the restriction map
$\alg{h}^0\to\alg{a}^*$ is surjective. Hence we may choose $\lambda_0\in\alg{h}^0$ extending
$\mu$, that is, $\lambda_0|_{\alg{a}}=\mu$. Define
\[ \eta(X)=\mu([X,v])\ (X\in\alg{g}),\qquad \lambda_t=\lambda_0+t\eta\ (t\in\RR). \]
The ideal $\alg{a}$ is abelian, so $\eta|_{\alg{a}}=0$. Also $\eta|_{\alg{h}}=0$, since $b$ is
$\mathrm{Ad}(H)$-invariant. Thus $\eta\in\alg{h}^0$ and $\lambda_t\in\alg{h}^0$ with
$\lambda_t|_{\alg{a}}=\mu$ for all $t$.

Apply the covector form \eqref{eq:covector} to $\lambda_t$, there exists
$Z\in\alg{h}$ such that
\[ 0=\lambda_t\big([\mathcal{A}\lambda_t+Z,v]\big). \]
Both $[\mathcal{A}\lambda_t,v]$ and $[Z,v]$ lie in the ideal $\alg{a}$. Hence
$\lambda_t([Z,v])=\mu([Z,v])=b(v,[Z,v])=0$. Consequently
\begin{equation}\label{eq:lem23}
0=\lambda_t\big([\mathcal{A}\lambda_t,v]\big).
\end{equation}
Now $[\mathcal{A}\lambda_t,v]\in\alg{a}$ and $\lambda_0|_{\alg{a}}=\mu$, so
$\lambda_0([\mathcal{A}\lambda_t,v])=\mu([\mathcal{A}\lambda_t,v])=\eta(\mathcal{A}\lambda_t)$. Also
$\eta([\mathcal{A}\lambda_t,v])=\mu([[\mathcal{A}\lambda_t,v],v])=0$ because $(\mathrm{ad}\,v)^2=0$ on $\alg{a}$.
Using $\lambda_t=\lambda_0+t\eta$ and the linearity of $\mathcal{A}$, equation \eqref{eq:lem23} gives
\[ 0=\lambda_t\big([\mathcal{A}\lambda_t,v]\big)=\eta(\mathcal{A}\lambda_t)
=\eta(\mathcal{A}\lambda_0)+t\,\eta(\mathcal{A}\eta)
\quad(t\in\RR). \]
Taking $t=0$ yields $\eta(\mathcal{A}\lambda_0)=0$, and then $t\,\eta(\mathcal{A}\eta)=0$ for all $t$, so
$\eta(\mathcal{A}\eta)=0$. But $\mathcal{A}\eta\in\Delta$ is the horizontal velocity of the covector $\eta$, and by
\eqref{eq:cometric},
\[ \eta(\mathcal{A}\eta)=Q(\mathcal{A}\eta,\mathcal{A}\eta)=Q^{*}(\eta|_{\Delta},\eta|_{\Delta})=0. \]
Thus $\eta|_{\Delta}=0$, that is, $b(v,[X,v])=\eta(X)=0$ for all $X\in\Delta$ and
$v\in\alg{a}$. Polarizing in $v$, we get $b(v,[X,w])+b(w,[X,v])=0$, so
$\mathrm{ad}\,X|_{\alg{a}}$ is skew-adjoint with respect to $b$ for every $X\in\Delta$. The same
holds for $X\in\alg{h}$ by the $\mathrm{Ad}(H)$-invariance of $b$. Since $\alg{h}+\Delta$ generates
$\alg{g}$ and the skew-adjoint endomorphisms of $\alg{a}$ form a Lie algebra, it follows that
$\mathrm{ad}\,X|_{\alg{a}}$ is skew-adjoint for all $X\in\alg{g}$; that is, $b$ is
$\mathrm{ad}(\alg{g})$-invariant.

Finally, if $X\in\alg{n}$, then $\mathrm{ad}\,X$ on $\alg{g}$ is nilpotent. Its restriction to
$\alg{a}$ is both nilpotent and skew-adjoint, hence zero. Thus $[\alg{n},\alg{a}]=0$.
\end{proof}

\begin{theorem}\label{thm:nilradical}
Suppose that $(G/H,\Delta,\langle\cdot,\cdot\rangle_\Delta)$ is geodesic orbit with respect to
$G$. Then $[\alg{n},[\alg{n},\alg{n}]]=0$.
\end{theorem}

\begin{proof}
Write the lower central series $C^1\alg{n}=\alg{n}$, $C^{j+1}\alg{n}=[\alg{n},C^j\alg{n}]$.
Suppose the nilpotency class were $r\ge3$, that is, $C^{r+1}\alg{n}=0$ but $C^r\alg{n}\neq0$.
Then the ideal $\alg{a}=C^{r-1}\alg{n}$ would be abelian, since
\[ [C^{r-1}\alg{n},C^{r-1}\alg{n}]\subset C^{2r-2}\alg{n}=0. \]
It is an ideal of $\alg{g}$ because the terms of the lower central series are characteristic in
$\alg{n}$ and $\alg{n}$ is an ideal of $\alg{g}$. Lemma~\ref{lem:abelian} gives
$[\alg{n},\alg{a}]=0$, contradicting $[\alg{n},C^{r-1}\alg{n}]=C^r\alg{n}\neq0$. Hence $r\le2$,
that is, $[\alg{n},[\alg{n},\alg{n}]]=0$.
\end{proof}

The following example shows that $\Delta\cap\alg{n}$ may vanish.

\begin{example}\label{ex:su2}
Let $\alg{g}=\alg{su}(2)\oplus\RR T\oplus\RR S$, with
$\alg{su}(2)=\operatorname{span}\{E_1,E_2,K\}$ and cyclic bracket $[E_1,E_2]=K$. Make these five
vectors orthonormal for a bi-invariant inner product $B$, and set
\[ \alg{h}=\RR(K+T),\qquad V=K-T,\qquad L=V+S,\qquad W=V-2S, \]
\[ \Delta=\operatorname{span}\{E_1,E_2,L\},\qquad Q=B|_{\Delta}. \]
This is a compact homogeneous model $(\mathrm{SU}(2)\times T^2)/S^1$, with compatible periods for
the diagonal circle and a finite ineffective kernel removed. In $\alg{m}=\alg{h}^{\perp_B}$,
write $\xi=X+aL+bW$ with $X\in\operatorname{span}\{E_1,E_2\}$. Its cometric image is $X+aL$, and a
direct computation gives
\[ [X+aL+b(K+T),\xi]=0, \]
so the structure is GO. The bracket $[E_1,E_2]=K$ also shows
$\alg{h}+\Delta+[\Delta,\Delta]=\alg{g}$. But
\[ \Delta\cap\alg{su}(2)=\operatorname{span}\{E_1,E_2\},\qquad \Delta\cap\alg{n}=0, \]
so their sum is smaller than $\Delta$: the horizontal space need not split along the Levi
decomposition.
\end{example}

\section{Compact cometrics and the proof of Theorem~\ref{thm:main}}\label{sec:compact}

In this section we suppose that $G$ is compact semisimple. Fix a positive definite
$\mathrm{Ad}(G)$-invariant inner product $B$ on $\alg{g}$ and put $\alg{m}=\alg{h}^{\perp_B}$. An
invariant horizontal metric is equivalently specified by an $H$-equivariant, $B$-self-adjoint,
nonnegative operator $A:\alg{m}\to\alg{m}$ satisfying
\[ \Delta=\imd A,\qquad Q(Ap,Aq)=B(p,Aq). \]
This is well defined on the image and positive definite there. Conversely, every
invariant sub-Riemannian structure $(\Delta,Q)$ determines one such $A$ (extend $Q$ by zero on the
orthogonal complement $\Delta^{\perp}$). The operator must additionally satisfy bracket generation.

\begin{lemma}[Regularization]\label{lem:reg}
The cotangent GO condition is
\begin{equation}\label{eq:GO}
\text{for every }p\in\alg{m}\text{ there is }Z\in\alg{h}\text{ such that }[p,Ap+Z]=0.
\end{equation}
For every $t>0$, this holds for $A$ if and only if it holds for $A+tI$. Hence $(\Delta,Q)$ is GO
if and only if the positive definite Riemannian metric with metric operator $(A+tI)^{-1}$ is
$G$-GO.
\end{lemma}

\begin{proof}
Represent $\lambda\in\alg{h}^0$ by $p\in\alg{m}$ using $B$, that is, $\lambda=B(p,\cdot)$. Then
by the ad-invariance of $B$, we obtain for all $W\in\alg{g}$
\[ \lambda([A\lambda+Z,W])=B(p,[Ap+Z,W])=B([p,Ap+Z],W).\]
It follows that  \eqref{eq:covector} is equivalent to \eqref{eq:GO}, since $B$ is nondegenerate.

Finally,
\[ [p,(A+tI)p+Z]=[p,Ap+t\,p+Z]=[p,Ap+Z]+t\,[p,p]=[p,Ap+Z], \]
so the same isotropy vector $Z$ works on both sides for every real $t$. Since $A+tI$ is positive
definite for $t>0$, it defines a Riemannian metric (with cometric $A+tI$ and metric operator
$(A+tI)^{-1}$) whose GO condition is exactly \eqref{eq:GO} for $A+tI$.
\end{proof}

\begin{remark}\label{rem:reg}
Regularization changes the horizontal metric, so it is not a Riemannian extension in the sense of
Section~\ref{sec:geod}.
\end{remark}

We fix the following notation. Write $\alg{g}=\bigoplus_{i=1}^{N}\alg{g}_i$ as a sum of compact
simple ideals and let $\rho_i:\alg{h}\to\alg{g}_i$ be the isotropy projection. Partition the
indices into
\[ I=\{i:\rho_i=0\},\qquad P=\{i:0\neq\rho_i(\alg{h})\subsetneq\alg{g}_i\},\qquad
F=\{i:\rho_i(\alg{h})=\alg{g}_i\}, \]
and put $J=P\cup F$ and $r=|J|$. We call $I$ the inactive set, $P$ the proper active set, and $F$
the full set. Each active $\rho_i$ is injective, hence an isomorphism onto its image. Let
$B_{\alg{h}}$ be the negative Killing form of $\alg{h}$ and choose positive invariant forms $B_i$
on $\alg{g}_i$ such that $\rho_i^{*}B_i=B_{\alg{h}}$ for $i\in J$; on inactive factors choose any
positive invariant normalization. Define
\begin{equation}\label{eq:compression}
B=\sum_i B_i,\qquad \alg{m}=\alg{h}^{\perp_B},\qquad C|_{\alg{g}_i}=c_i\Idm,\qquad
A_c=P_{\alg{m}}(C|_{\alg{m}}).
\end{equation}
Inactive factors are full-rank Riemannian factors carrying bi-invariant metrics, since a
left-invariant metric on a compact semisimple Lie group is geodesic orbit need to be bi-invariant; such factors therefore
split off as direct product factors. Hence we may assume that $I=\varnothing$ without loss of
generality. There are two cases: $r=1$ and $r\ge2$, treated in Sections~\ref{subsec:simple}
and~\ref{subsec:semisimple}, respectively.

\subsection{The simple case: proof of Theorem~\ref{thm:main}(A)}\label{subsec:simple}

Here $r=1$, so $\alg{g}$ is simple. By Lemma~\ref{lem:reg}, for every $t>0$ the positive definite
operator $A+tI$ defines a $G$-GO Riemannian metric $\mu_t$ on $G/H$, with cometric $A+tI$ and
metric operator $(A+tI)^{-1}$. By the Riemannian classification of
Chen--Nikolayevsky--Nikonorov \cite[Thms.~1--2]{ChenNN}, a $G$-GO Riemannian metric on $G/H$ with
$\alg{g}$ and $\alg{h}$ simple is either naturally reductive or belongs to the exceptional
(non-naturally-reductive) list of \cite[Table~1]{ChenNN}. Since naturally reductive metrics on
$G/H$ are unique up to scale, a proper $G$-GO structure must come from the exceptional list.

We now fix a Lie algebra pair $(\alg{g},\alg{h})$ from \cite[Table~1]{ChenNN}, together with a
$B$-orthogonal decomposition
\[ \alg{g}=\alg{h}\oplus\alg{m}_1\oplus\cdots\oplus\alg{m}_s \]
into irreducible $\alg{h}$-submodules $\alg{m}_i$ of $\alg{m}$, and write the cometric as
$A=\bigoplus_i\beta_i\Idm|_{\alg{m}_i}$ with $\beta_i\ge0$. The horizontal space is then
$\Delta=\bigoplus_{\beta_i>0}\alg{m}_i$, and bracket generation means
$\alg{h}+\Delta+[\Delta,\Delta]=\alg{g}$. By Lemma~\ref{lem:reg}, the sub-Riemannian metric with
cometric $A$ is GO if and only if, for some $t>0$, the positive invariant Riemannian metric
$\sum_i\alpha_i\,B|_{\alg{m}_i}$ on $G/H$ is GO, where $\alpha_i=(\beta_i+t)^{-1}$. We check these
conditions case by case.

\textbf{Case 1:} the two-summand rows of Table~\ref{tab:two}.
Suppose $s=2$, so $\alg{m}=\alg{m}_1\oplus\alg{m}_2$. In every two-summand row of
\cite[Table~1]{ChenNN} the positive invariant GO metrics form the full two-parameter family: both
coefficients $\alpha_1,\alpha_2$ are arbitrary positive numbers, and every positive invariant
metric is GO. On the boundary we may set $\beta_1=0$ or $\beta_2=0$.

If $\beta_1=0$ and $\beta_2>0$, then $\Delta=\alg{m}_2$. In each row
$\alg{h}\oplus\alg{m}_2$ is the proper intermediate subalgebra indicated in
\cite[Table~1]{ChenNN}: for example, in row 1 one has
$\alg{h}\oplus\alg{m}_2=\alg{so}(8)\subsetneq\alg{so}(9)$, and the same pattern holds in every
row. Hence $\alg{h}+\Delta+[\Delta,\Delta]\subset\alg{h}\oplus\alg{m}_2\subsetneq\alg{g}$, and
bracket generation fails.

If $\beta_2=0$ and $\beta_1>0$, then $\Delta=\alg{m}_1$, and we claim that $\Delta$ is bracket
generating. Set $\alg{l}=\alg{h}\oplus\alg{m}_2$; this is a proper subalgebra of $\alg{g}$, and
$\alg{m}_1=\alg{l}^{\perp_B}$. Since $B$ is $\mathrm{ad}$-invariant, for $u\in\alg{l}$ and
$v,w\in\alg{m}_1$ we have
\[ B([u,v],w)=-B(v,[u,w])=0, \]
because $[u,w]\in\alg{l}$ and $v\perp\alg{l}$; hence $[\alg{l},\alg{m}_1]\subset\alg{m}_1$, that
is, $\alg{m}_1$ is $\alg{l}$-stable. Let $\alg{a}$ be the Lie algebra generated by $\alg{m}_1$.
Since the adjoint action of $\alg{l}$ is by derivations and stabilizes the generators,
$[\alg{l},\alg{a}]\subset\alg{a}$; and since $\alg{a}$ is a subalgebra containing $\alg{m}_1$,
also $[\alg{m}_1,\alg{a}]\subset\alg{a}$. Writing $\alg{g}=\alg{l}\oplus\alg{m}_1$, we conclude
$[\alg{g},\alg{a}]\subset\alg{a}$: the algebra $\alg{a}$ is a nonzero ideal of the simple algebra
$\alg{g}$, hence $\alg{a}=\alg{g}$. Thus $\alg{h}\oplus\alg{m}_1$ generates $\alg{g}$, that is,
the structure with $\Delta=\alg{m}_1$ and $A=\beta_1\Idm|_{\alg{m}_1}$, $\beta_1>0$, is bracket
generating and GO. This proves that the two-summand rows contribute exactly the structures of
Table~\ref{tab:two}, and nothing else.

\textbf{Case 2:} row $(6_{1})$ with $n$ odd.
Here $G/H=\mathrm{Spin}(2n+1)/\mathrm{SU}(n)$, $n\ge3$ odd, with $\alg{h}=\alg{su}(n)$ and
\[ \alg{m}_1=\alg{so}(2n{+}1)\ominus\alg{so}(2n),\qquad
\alg{m}_2=\alg{so}(2n)\ominus\alg{u}(n),\qquad
\alg{m}_3=\alg{u}(n)\ominus\alg{su}(n), \]
three summands, all contained in $\alg{so}(2n)$ except $\alg{m}_1$. The GO condition for the
positive metrics of this row, expressed in cometric coefficients, is the single relation
\begin{equation}\label{eq:row6}
n\beta_3=(n-1)\beta_2+\beta_1.
\end{equation}
Bracket generation requires $\beta_1>0$: if $\beta_1=0$, then
$\Delta\subset\alg{m}_2\oplus\alg{m}_3\subset\alg{so}(2n)$, and since also
$\alg{h}=\alg{su}(n)\subset\alg{so}(2n)$, we get
$\alg{h}+\Delta+[\Delta,\Delta]\subset\alg{so}(2n)\subsetneq\alg{so}(2n+1)$, so generation fails.
Properness then forces $\beta_2=0$: indeed, if $\beta_2>0$, then \eqref{eq:row6} gives
$\beta_3=\big((n-1)\beta_2+\beta_1\big)/n>0$, so all three coefficients are positive and $A$ is
positive definite (full rank), contradicting properness; note that $\beta_3=0$ is incompatible
with $\beta_1>0$, since it would give $\beta_1=-(n-1)\beta_2\le0$. With $\beta_2=0$,
\eqref{eq:row6} gives $\beta_3=\beta_1/n>0$, so the boundary structures of this row are exactly
$(\beta_1,0,\beta_1/n)$, $\beta_1>0$, as displayed in Table~\ref{tab:rem}. These conditions are
also sufficient, by Lemma~\ref{lem:reg}.

When $n=3$, the first two $\mathrm{SU}(3)$-modules $\alg{m}_1$ and $\alg{m}_2$ are equivalent; the
GO condition nevertheless forces the displayed splitting; see \cite[Sec.~5.2.3]{ChenNN}.

\textbf{Case 3:} row (9).
Here $G/H=\mathrm{SU}(2n+1)/\mathrm{Sp}(n)$, $n\ge2$, with $\alg{h}=\alg{sp}(n)$ and
\[ \alg{m}_1=\alg{su}(2n{+}1)\ominus\alg{u}(2n),\qquad
\alg{m}_2=\alg{su}(2n)\ominus\alg{sp}(n),\qquad
\alg{m}_3=\alg{u}(2n)\ominus\alg{su}(2n). \]
The positive coefficients of the source metrics are unrestricted: every positive invariant metric
on this space is GO. Bracket generation is equivalent to $\beta_1>0$. Indeed,
$\alg{h}+\alg{m}_2+\alg{m}_3=\alg{u}(2n)$ is a proper subalgebra of $\alg{su}(2n+1)$, and its
orthogonal complement is $\alg{m}_1$; if $\beta_1=0$, then $\Delta\subset\alg{u}(2n)$ and
generation fails, while if $\beta_1>0$, then $\Delta$ contains $\alg{m}_1$, which generates
$\alg{su}(2n+1)$ by the argument of Case 1. Properness is exactly $\beta_2\beta_3=0$: if
$\beta_2\beta_3>0$, then together with $\beta_1>0$ all three coefficients are positive and $A$ is
full rank; conversely, if $\beta_2\beta_3=0$, then $A$ is nonzero (since $\beta_1>0$) and
rank-deficient, hence proper.

\textbf{Case 4:} row (10).
Here $G/H=\mathrm{Spin}(8)/G_2$, and the two irreducible isotropy summands are equivalent
seven-dimensional $\alg{g}_2$-modules (they are interchanged by an outer automorphism of
$\alg{so}(8)$). For this reason the Riemannian classification \cite[Table~1]{ChenNN} permits every
positive invariant operator, so every nonnegative invariant operator is GO by Lemma~\ref{lem:reg}.
A proper nonzero one has rank seven, and we now determine its form and its bracket-generation
condition explicitly.

Identify $\alg{so}(8)=\mathrm{Der}(\OO)\oplus L_{\mathrm{Im}\,\OO}\oplus R_{\mathrm{Im}\,\OO}$, where
$\alg{h}=\alg{g}_2=\mathrm{Der}(\OO)$ and $L_x$, $R_x$ denote left and right multiplication by $x$ on
the division algebra of octonions $\OO$; the two isotropy summands are $L_{\mathrm{Im}\,\OO}$ and
$R_{\mathrm{Im}\,\OO}$, both isomorphic to the seven-dimensional irreducible $\alg{g}_2$-module.
By Schur's lemma, every $\alg{h}$-equivariant self-adjoint endomorphism of this irreducible module
is a scalar, so the $\alg{h}$-invariant seven-dimensional subspaces of
$L_{\mathrm{Im}\,\OO}\oplus R_{\mathrm{Im}\,\OO}$ are exactly the graphs
\[ V_{a,b}=\{aL_x+bR_x:x\in\mathrm{Im}\,\OO\},\qquad [a:b]\in\RR P^1, \]
and a proper nonzero invariant operator is $A=\beta\Idm|_{V_{a,b}}$ with $\beta>0$.

To check bracket generation, write $c=[x,y]=xy-yx$ (the commutator, which lies in
$\mathrm{Im}\,\OO$) and
\[ D_{x,y}(z)=[[x,y],z]-\tfrac13\big((xy)z-x(yz)\big); \]
this is an octonion derivation, hence lies in $\mathrm{Der}(\OO)=\alg{h}$; see \cite[Sec.~4.1]{Baez}.
Using alternativity one computes, for $x,y\in\mathrm{Im}\,\OO$,
\[ 3[aL_x+bR_x,aL_y+bR_y]=(a^2+2ab-2b^2)L_c+(2a^2-2ab-b^2)R_c+2(a^2+b^2-ab)D_{x,y}. \]
Hence, modulo $\alg{h}$, the bracket $[V_{a,b},V_{a,b}]$ is spanned by the vectors
$\alpha L_c+\beta R_c$ with $c$ ranging over $\mathrm{Im}\,\OO$, where $\alpha=a^2+2ab-2b^2$ and
$\beta=2a^2-2ab-b^2$. Since the commutators $[x,y]$ span $\mathrm{Im}\,\OO$, the spaces
$\{aL_x+bR_x\}$ and $\{\alpha L_x+\beta R_x\}$ together span
$L_{\mathrm{Im}\,\OO}\oplus R_{\mathrm{Im}\,\OO}$ exactly when the coefficient matrix
$\bigl(\begin{smallmatrix} a & b\\ \alpha & \beta \end{smallmatrix}\bigr)$ has rank two, that is,
when
\[ a\beta-b\alpha=a(2a^2-2ab-b^2)-b(a^2+2ab-2b^2)=(a+b)(2a-b)(a-2b)\neq0. \]
In this case $\alg{h}+V_{a,b}+[V_{a,b},V_{a,b}]=\alg{so}(8)$ and the structure is bracket
generating. If the determinant vanishes, then $(\alpha,\beta)$ is proportional to $(a,b)$, so
$\{\alpha L_x+\beta R_x\}=V_{a,b}$ and $[V_{a,b},V_{a,b}]\subset\alg{h}\oplus V_{a,b}$; thus
$\alg{h}+V_{a,b}$ is already a proper Lie subalgebra and generation fails. This proves the three
exclusions $[1:-1]$, $[1:2]$, $[2:1]$ of Table~\ref{tab:rem} and completes the analysis of row (10).
\qed

\subsection{The semisimple case: proof of Theorem~\ref{thm:main}(B)}\label{subsec:semisimple}

Here $r\ge2$, so $\alg{h}$ projects nontrivially onto at least two simple ideals. For $i\in P$,
write $\alg{g}_i=\alg{h}_i\oplus\alg{p}_i$, where $\alg{h}_i=\rho_i(\alg{h})$ and the sum is
$B_i$-orthogonal. Identify active copies of $\alg{h}$ using $\rho_i$. The orthogonal decomposition
is
\begin{equation}\label{eq:decomp}
\alg{m}=\bigoplus_{i\in P}\alg{p}_i\,\oplus\,\Bigl\{(z_i)_{i\in J}:\sum_{i\in J}z_i=0\Bigr\}.
\end{equation}
The quadratic form of $A_c$ is $c_iB_i$ on the summands $\alg{p}_i$ and
\[ q_c(z)=\sum_{i\in J}c_i|z_i|^2,\qquad \sum_{i\in J}z_i=0, \]
on the diagonal summand, where $|\cdot|$ is induced by $B_{\alg{h}}$. Consequently, on its image,
the sub-Riemannian metric is $Q=\big(\sum_i c_i^{-1}B_i\big)\big|_{\imd A_c}$.

\begin{lemma}\label{lem:compGO}
With notation as above, for any coefficients $c_i$ the sub-Riemannian metric defined by $A_c$ is
geodesic orbit.
\end{lemma}

\begin{proof}
For every $p\in\alg{m}$, take $Z=P_{\alg{h}}Cp$. Then
\[ A_cp+Z=Cp,\qquad [p,A_cp+Z]=[p,Cp]=0. \]
Thus every nonnegative cometric is GO whenever its image is bracket generating.
\end{proof}

\begin{lemma}\label{lem:form}
With notation as above, a nonnegative GO operator $A:\alg{m}\to\alg{m}$ must be of the form
\eqref{eq:compression}.
\end{lemma}

\begin{proof}
By Lemma~\ref{lem:reg}, for every $t>0$ the Riemannian metric with metric operator
$(A+tI)^{-1}$ is $G$-GO. Hence it is $G$-naturally reductive and has the form
\eqref{eq:compression}, by Theorems 1 and 2 of \cite{ChenNN}.
\end{proof}

\begin{lemma}\label{lem:quadform}
Let $V\neq0$ be a real inner product space and $r\ge2$. The form $\sum_{i=1}^{r}c_i|z_i|^2$ on
$\sum_i z_i=0$ is nonnegative if and only if either all $c_i\ge0$, or exactly one coefficient $c_j$
is negative, every other coefficient is positive, and
\[ c_j\ge-\Bigl(\sum_{i\neq j}c_i^{-1}\Bigr)^{-1}. \]
In the first case its kernel consists of vectors supported on the zero coefficients and having
coordinate sum zero. In the second case it is positive definite for strict inequality; at equality
its kernel is $\{(c_i^{-1}u)_i:u\in V\}$.
\end{lemma}

\begin{proof}
Two negative coefficients, or one negative and one zero coefficient, give a negative value on an
opposite pair of nonzero coordinates. In the all-nonnegative case the assertion about the kernel
follows term by term. In the remaining case put $S=\sum_{i\neq j}c_i^{-1}>0$. Subject to
$\sum_{i\neq j}z_i=-z_j$, the weighted Cauchy--Schwarz inequality gives
\[ \sum_{i\neq j}c_i|z_i|^2\ge\frac{|z_j|^2}{S}, \]
with equality exactly when $z_i=-z_j/(Sc_i)$. Hence nonnegativity is equivalent to
$c_j+1/S\ge0$, and the equality case gives the stated kernel. For strict inequality, vanishing
requires $z_j=0$ and then every other coordinate zero.
\end{proof}

\begin{theorem}\label{thm:criterion}
Under the preceding hypotheses, the three properties
\[ A_c\ge0,\qquad \Delta=\imd A_c\neq\alg{m},\qquad \alg{g}=\alg{h}+\Delta+[\Delta,\Delta] \]
hold simultaneously if and only if the coefficients satisfy, for some $j\in F$,
\begin{equation}\label{eq:criterion}
c_i>0\ (i\neq j),\qquad c_j=-\Bigl(\sum_{i\in J\setminus\{j\}}c_i^{-1}\Bigr)^{-1}<0,\qquad
P\neq\varnothing\ \text{or}\ r\ge3.
\end{equation}
Moreover, the kernel is
\begin{equation}\label{eq:kernel}
\Ker A_c=\{n(u):u\in\alg{h}\},\qquad
n(u)_i=c_i^{-1}\rho_i(u),\quad i\in J.
\end{equation}
Consequently the corank is $\dim\alg{h}$ and the horizontal rank is
$\dim\alg{g}-2\dim\alg{h}$.
\end{theorem}

\begin{proof}
\emph{Necessity.} Nonnegativity and bracket generation force $c_i>0$ for every $i\in P$. Indeed, a
negative coefficient makes $A_c$ negative on the nonzero summand $\alg{p}_i$ in
\eqref{eq:decomp}. If such a coefficient is zero, the projection of $\alg{h}+\imd A_c$ to
$\alg{g}_i$ is contained in the proper subalgebra $\alg{h}_i$; brackets cannot generate
$\alg{g}_i$.

If all active coefficients are nonnegative, singularity requires at least two zero coefficients,
and their indices must be full. Choose distinct such indices $a,b$. For every $u\in\alg{h}$, the
vector with active coordinates $z_a=u$, $z_b=-u$ and all others zero belongs to $\Ker A_c$. Since
$A_c$ is self-adjoint, its image is orthogonal to this kernel, and therefore its $a$ and $b$
coordinates agree under the identifications with $\alg{h}$. The same is true of every element of
$\alg{h}$. The equality of these two coordinates defines a proper Lie subalgebra of $\alg{g}$,
contradicting bracket generation.

There is consequently one negative active coefficient, at some $j\in F$. Lemma~\ref{lem:quadform}
implies that every other active coefficient is positive; hence $c_i>0$ for all $i\neq j$. The
strict inequality of that lemma would make $A_c$ positive definite. Singularity therefore forces
\[ c_j=-1/S,\qquad S=\sum_{i\in J\setminus\{j\}}c_i^{-1},\qquad \sum_{i\in J}c_i^{-1}=0. \]
The kernel is precisely \eqref{eq:kernel}.

If $r=2$ and $P=\varnothing$, both active factors are full and the last summand in
\eqref{eq:decomp} is the entire kernel. Thus $\Delta$ has no active component. The projection of
$\alg{h}+\Delta$ onto the two active factors is diagonal, so it cannot generate their direct sum.
Necessity of the final condition in \eqref{eq:criterion} follows.

\emph{Sufficiency.} Assume \eqref{eq:criterion} holds. Lemma~\ref{lem:quadform} proves
nonnegativity, singularity and \eqref{eq:kernel}. Write $N=\Ker A_c$; since
$(\alg{h}+\Delta)^{\perp_B}=N$, it suffices to prove that a vector $n(u)\in N$ orthogonal to
$[\Delta,\Delta]$ must vanish.

First suppose $P\neq\varnothing$ and choose $k\in P$. We have $\alg{p}_k\subset\Delta$. The span
of the $\alg{h}_k$-projections of $[\alg{p}_k,\alg{p}_k]$ is an $\alg{h}_k$-ideal, by equivariance.
It is nonzero: otherwise invariance of $B_k$ would give
\[ B_k([U,X],Y)=B_k(U,[X,Y])=0\qquad(U\in\alg{h}_k,\ X,Y\in\alg{p}_k). \]
Because $[\alg{h}_k,\alg{p}_k]\subset\alg{p}_k$, this would imply $[\alg{h}_k,\alg{p}_k]=0$ and
make $\alg{h}_k$ a nonzero proper ideal of the simple $\alg{g}_k$. The span is therefore all of
$\alg{h}_k$. Orthogonality of $n(u)$ to these brackets says
\[ 0=B(n(u),[X,Y])=c_k^{-1}B_k(\rho_k(u),[X,Y])\qquad(X,Y\in\alg{p}_k), \]
which forces $u=0$.

Now suppose $P=\varnothing$, so $r\ge3$ and all active factors are full. Put $v_i=c_i^{-1}$ for
$i\in F$ and choose a nonzero real vector
\[ a\in W=\Bigl\{(a_i)_{i\in F}:\sum_i a_i=0,\ \sum_i v_ia_i=0\Bigr\}. \]
This is possible because $\dim W=r-2$; the vectors $(1)_i$ and $(v_i)_i$ are independent. For
every $x\in\alg{h}$, the vector $a\otimes x:=(a_i\rho_i(x))_{i\in F}$ lies in
$\Delta=\alg{m}\cap N^{\perp_B}$. The identities $\sum_i v_i=\sum_i v_ia_i=0$ give
\[ \kappa:=\sum_i v_ia_i^2=\sum_{i\neq j}v_i(a_i-a_j)^2>0. \]
Strict positivity follows because all $v_i$ for $i\neq j$ are positive; equality would make all
$a_i$ equal, and $\sum_i a_i=0$ would then imply $a=0$. For $x,y\in\alg{h}$,
\[ B\big(n(u),[a\otimes x,a\otimes y]\big)=\kappa B_{\alg{h}}(u,[x,y]). \]
Since $[\alg{h},\alg{h}]=\alg{h}$, orthogonality forces $u=0$. This proves
$\alg{h}+\Delta+[\Delta,\Delta]=\alg{g}$ in both cases. The metric is positive definite on the
image, and formula \eqref{eq:kernel} gives the rank and corank.
\end{proof}

Theorem \ref{thm:main}(B) is indeed the following theorem.
\begin{theorem}\label{thm:isom}
Under the above hypotheses. Let $(G/H,\Delta,\langle\cdot,\cdot\rangle_\Delta)$ be a compact, connected, simply connected
 sub-Riemannian geodesic orbit manifold with $G$ semisimple and $H$ simple. Then it is isometric to
$((K\times H)/\diag H,\Delta,B|_{\Delta})=(K,\Delta,b)$, where $K$ is a compact simply connected semisimple Lie group,  $b$ is a positive bi-invariant metric on $K$, and
$\Delta=\alg{h}^{\perp_b}\subset\alg{k}$ is the $b$-orthogonal complement of the Lie algebra
$\alg{h}$ of $H$ in the Lie algebra $\alg{k}$ of $K$,
$H\simeq (\rho_{i}(H))_{i\neq j}\subset K$ is a
connected simple subgroup that is not normal.
\end{theorem}

\begin{proof}
Notation as in Theorem~\ref{thm:criterion}. Put $\alg{k}=\bigoplus_{i\neq j}\alg{g}_i$ and
$b=\sum_{i\neq j}c_i^{-1}B_i$. The remaining projection embeds $\alg{h}$ into $\alg{k}$. At the
group level, replace $H$ by its image in $K$, dividing out the possible finite ineffective kernel.
The homogeneous quotient $(K\times H)/\diag H$ is $K$, with tangent map
\[ T\big((x_i)_{i\neq j},z_j\big)=\big(x_i-\rho_i(z_j)\big)_{i\neq j}, \]
where $\rho_i$ is the isotropy embedding. For $p\in\alg{m}$ with active
components $z_i$, orthogonality to $\Ker A_c$ reads
\[ \sum_{i\neq j}z_i/c_i-Sz_j=0. \]
This is precisely $b(Tp,\alg{h})=0$. Also $T(A_cp)_i=c_ip_i-c_j\rho_i(z_j)$, and
$\sum_{i\neq j}z_i=-z_j$, $c_jS=-1$ give
\[ b\big(T(A_cp),T(A_cp)\big)=\sum_{i\neq j}c_i|p_i|^2+(2c_j+c_j^2S)|z_j|^2=B(p,A_cp). \]
Polarization proves that the resulting horizontal metric is exactly $Q=b|_{\Delta}$.

It remains to determine bracket generation. The Lie algebra generated by $\alg{h}^{\perp_b}$ is
$\alg{h}$-invariant, hence is an ideal of $\alg{k}=\alg{h}+\alg{h}^{\perp_b}$. Its
$b$-orthogonal complement is also an ideal and lies in $\alg{h}$. Simplicity of $\alg{h}$ makes
this complement either zero or $\alg{h}$. The latter occurs exactly when $\alg{h}$ is an ideal of
$\alg{k}$. This proves the stated condition that $H$ be non-normal.
\end{proof}

\begin{example}\label{ex:three}
Let $\alg{g}=\alg{h}\oplus\alg{h}\oplus\alg{h}$, with diagonally embedded compact simple $\alg{h}$,
equal normalized forms, and $(c_1,c_2,c_3)=(1,1,-\tfrac12)$. Then
\[ \Ker A_c=\{(u,u,-2u):u\in\alg{h}\},\qquad \Delta=\{(x,-x,0):x\in\alg{h}\}. \]
On $\Delta$, $A_c=I$ and $Q=B|_{\Delta}$. Brackets of horizontal vectors have the form
$([x,y],[x,y],0)$; together with $\Delta$ and the diagonal isotropy they span all of $\alg{g}$.
\end{example}

\begin{example}\label{ex:two}
Let $\rho:\alg{h}\hookrightarrow\alg{g}_1$ be a proper embedding into a compact simple algebra,
and set $\alg{g}=\alg{g}_1\oplus\alg{h}$ with isotropy $\{(\rho(u),u):u\in\alg{h}\}$. With
normalized forms, take $(c_1,c_2)=(a,-a)$, $a>0$. Then $\Delta=\alg{p}_1\oplus0$, and
$Q=a^{-1}B_1$ on this subspace. This is an admissible two-active-factor case because the first
projection is proper. If instead $\alg{g}_1=\alg{h}$ and $\rho$ is an isomorphism, then $A_c=0$ on
the active tangent space and bracket generation fails.
\end{example}

\section*{Declarations}\label{sec:declarations}

\noindent\textbf{Acknowledgments.} H. An is supported by the special fund for basic scientific
research expenses of universities in Liaoning Province (LJ212410165012). Z. Yan is supported by the
Zhejiang Provincial Natural Science Foundation of China under Grant No. LMS25A010010. H. Zhang is
supported by NSFC (Grant No. 12301032), NSF of Jiangsu (Grant No. BK20230803) and Zhishan Scholars
Programs of Southeast University (Grant No. 2242024RCB0039). S. Zhang is partially supported by the
Natural Science Foundation of Shandong Province (No. ZR2026MS0024) and by the Science and
Technology Support Plan for Youth Innovation of Colleges and Universities of Shandong Province of
China (No. 2023KJ090).

\noindent\textbf{AI Disclosure.} Generative artificial intelligence tools, specifically ChatGPT and
WorkBuddy, were used during the preparation of this manuscript as auxiliary aids for language
editing, improving exposition and organization. All mathematical statements, proofs,
computations, references, and conclusions were independently verified by the authors, who take full
responsibility for the content of the manuscript.

\noindent\textbf{Data availability statement.} This manuscript has no associated data.

\noindent\textbf{Conflicts of interest.} The authors declare that they have no conflict of
interest.

\end{document}